\documentclass[10pt]{amsart}

\usepackage[utf8x]{inputenc}
\usepackage{subfigure}
\usepackage[english]{babel}
\usepackage{amsmath, amsfonts, amssymb}
\usepackage{color}
\usepackage{tikz}
\usetikzlibrary{positioning,shapes,shadows,arrows}
\usepackage{enumitem}
\usepackage{hhline}
\usepackage{array}

\newtheorem{Theoreme}{Theorem}[section]

\newtheorem{Proposition}[Theoreme]{Proposition}

\newtheorem{Definition}[Theoreme]{Definition}

\newcommand{\trprec}{\vartriangleleft} 
\newcommand{\ntrprec}{\ntriangleleft}

\DeclareMathOperator{\Itrees}{trees} 
\DeclareMathOperator{\Isize}{size} 
\DeclareMathOperator{\Iir}{ir} 

\newcommand\dec{F_{\ge}}
\newcommand\inc{F_{\le}}

\tikzstyle{Red} = [color = red]
\tikzstyle{Blue} = [color = blue]
\tikzstyle{Green} = [color = darkGreen]
\tikzstyle{Gray} = [color = gray]

\tikzstyle{Path} = [line width = 1.2]
\tikzstyle{StrongPath} =  [line width=2.5]
\tikzstyle{DPoint} = [fill, radius=0.1]

\tikzstyle{Point} = [fill, radius=0.08]

\author{Fr\'ed\'eric Chapoton, Gr\'egory Ch\^atel, Viviane Pons}

\address{
Institut Camille Jordan, Univ. Claude Bernard Lyon \\
Laboratoire d'Informatique Gaspard Monge, Univ. Paris-Est Marne-la-Vall\'ee \\ 
Fakult\"at f\"ur Mathematik, Universit\"at Wien
    }
    
\title{Two bijections on Tamari Intervals}


\begin{document}
\maketitle

\begin{abstract}
We use a recently introduced combinatorial object, the \emph{interval-poset}, to describe two bijections on intervals of the Tamari lattice. Both bijections give a combinatorial proof of some previously known results. The first one is an inner bijection between Tamari intervals that exchanges the \emph{initial rise} and \emph{lower contacts} statistics. Those were introduced by Bousquet-Mélou, Fusy, and Préville-Ratelle who proved they were symmetrically distributed but had no combinatorial explanation. The second bijection sends a Tamari interval to a closed flow of an ordered forest. These combinatorial objects were studied by Chapoton in the context of the Pre-Lie operad and the connection with the Tamari order was still unclear.

Nous utilisons les \emph{intervalles-posets}, très récemment introduits, pour décrire deux bijections sur les intervalles du treillis de Tamari. Nous obtenons ainsi des preuves combinatoires de précédents résultats. La première bijection est une opération interne sur les intervalles qui échange les statistiques de la \emph{montée initiale} et du \emph{nombre de contacts}. Ces dernières ont été introduites par Bousquet-Mélou, Fusy et Préville-Ratelle qui ont prouvé qu'elles étaient symétriquement distribuées sans pour autant proposer d'explication combinatoire. La seconde bijection fait le lien avec un objet étudié par Chapoton dans le cadre de l'opérade Pré-Lie : les flots sur les forêts ordonnées. Le lien avec l'ordre de Tamari avait déjà été remarqué sans pour autant être expliqué.
\end{abstract}

\section{Introduction}
The intervals of the Tamari lattice $T_{n}$ have recently been studied in various combinatorial and algebraic contexts. The first notable result was from Chapoton \cite{Chap} who proved that they were enumerated by a very nice formula, namely

\begin{equation}
\label{eq:intervals-formula}
\text{Number of intervals of $T_{n}$} = \frac{2}{n(n +1)} \binom{4 n + 1}{n - 1}.
\end{equation}

Note that this also counts the number of planar triangulations (\emph{i.e.}, maximal planar graphs) \cite{SchaefferTriang} and an explicit bijection was given by Bernardi and Bonichon \cite{BijTriangulations}. The formula itself was recently generalized to the $m$-Tamari lattices $T_{n}^{(m)}$ by Bousquet-Mélou, Fusy, and Préville-Ratelle~\cite{mTamari},

\begin{equation}
\label{eq:m-intervals-formula}
\text{Number of intervals of $T_{n}^{(m)}$} = \frac{m+1}{n(mn +1)} \binom{(m+1)^2 n + m}{n - 1}.
\end{equation} 

It is very remarkable than both formulas \eqref{eq:intervals-formula} and \eqref{eq:m-intervals-formula} have such simple factorized expressions. It convinces us that the combinatorics of intervals of the Tamari lattice is indeed very interesting and still has many properties to be discovered. In a very recent work \cite{Me_Tamari_FPSAC, Me_Tamari}, a subset of the authors of the present paper introduced a new object to this purpose: the interval-posets of Tamari. These are labelled posets which represent intervals of the Tamari lattice. The interval-posets were used to retrieve the functional equations leading to \eqref{eq:intervals-formula} and \eqref{eq:m-intervals-formula} and allowed for new enumeration results. In this paper, we intend to show how these new objects can be used to solve other open problems on Tamari intervals.

We first give a short summary of basic definitions and constructions in Section \ref{sec:interval-poset}. Section \ref{sec:initial-rise} is dedicated to an inner bijection on interval-posets. Thanks to it, we obtain a combinatorial proof of what was left as an open question in \cite{mTamari}: the symmetric distribution of the intial rise and lower contacts of intervals. Our bijection is based on two different recursive decompositions of interval-posets. 

In Section \ref{sec:flows}, we describe a bijection which sends interval-posets to flows of ordered forests. The flows of rooted trees appeared in  \cite{Flows} in an algebraic context and a surprising connection was made with the Tamari lattice by comparing some enumeration polynomials to those of \cite{Me_Tamari}. The explicit bijection we now present is a step forward into understanding the relations between the two theories. 

The present paper intends to be a summarized overview of both bijections of Section \ref{sec:initial-rise} and \ref{sec:flows}. We believe that both of them will lead us to more results and to a better understanding of the numerous combinatorial aspects of Tamari intervals. This should  be further explored in some future work.

\section{Interval-posets}
\label{sec:interval-poset}

\subsection{Tamari lattice}
\label{sub-sec:tamari-lattice}

The Tamari lattice is an order on Catalan objects which was first described by Tamari \cite{Tamari2} on formal parentheses. On binary trees, it can be seen as the transitive and reflexive closure of the right rotation operation. We recursively define a binary tree by being either an empty tree (or a leaf) or a pair of binary trees (resp. called \emph{left} and \emph{right} subtrees) grafted on a root. The size of a tree is the number of internal nodes. If a tree $T$ is composed of a root node $x$ with $A$ and $B$ as respectively left and right subtrees, we write $T = x(A,B)$. The right rotation on a node $y$ with a left child $x$ consists in replacing $y(x(A,B),C)$ by $x(A,y(B,C))$ where $A$, $B$, and $C$ are binary trees (possibly empty) as illustrated on Figure \ref{fig:tree-right-rotation}.

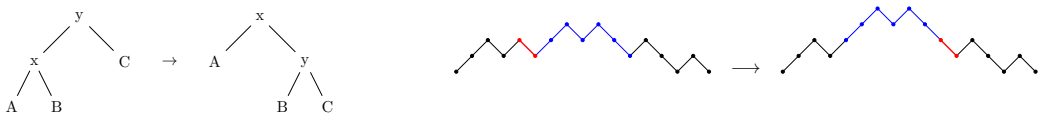
\begin{figure}[ht]
\centering
\begin{tabular}{ccc}
\scalebox{0.6}{
  \begin{tikzpicture}[baseline=0.3cm]
    \node(TX1) at (-3,0){x};
    \node(TY1) at (-2,1){y};
    \node(TA1) at (-3.5,-1){A};
    \node(TB1) at (-2.5,-1){B};
    \node(TC1) at (-1,0){C};
    \node(to) at (0,0){$\to$};
    \node(TX2) at (2,1){x};
    \node(TY2) at (3,0){y};
    \node(TA2) at (1,0){A};
    \node(TB2) at (2.5,-1){B};
    \node(TC2) at (3.5,-1){C};
    
    \draw (TA1) -- (TX1);
    \draw (TB1) -- (TX1);
    \draw (TC1) -- (TY1);
    \draw (TX1) -- (TY1);
    
    \draw (TA2) -- (TX2);
    \draw (TB2) -- (TY2);
    \draw (TC2) -- (TY2);
    \draw (TY2) -- (TX2);
  \end{tikzpicture}
} & \hspace{0.5cm} & \scalebox{0.7}{

\begin{tabular}{ccc}

\scalebox{0.3}{
\begin{tikzpicture}
\draw[Path] (0.000, 0.000) -- (1.000, 1.000);
\draw[Path] (1.000, 1.000) -- (2.000, 2.000);
\draw[Path] (2.000, 2.000) -- (3.000, 1.000);
\draw[Path] (3.000, 1.000) -- (4.000, 2.000);
\draw[DPoint] (0.000, 0.000) circle;
\draw[DPoint] (1.000, 1.000) circle;
\draw[DPoint] (2.000, 2.000) circle;
\draw[DPoint] (3.000, 1.000) circle;
\draw[StrongPath,Red] (4.000, 2.000) -- (5.000, 1.000);
\draw[DPoint,Red] (4.000, 2.000) circle;
\draw[Path,Blue] (5.000, 1.000) -- (6.000, 2.000);
\draw[DPoint,Red] (5.000, 1.000) circle;
\draw[Path,Blue] (6.000, 2.000) -- (7.000, 3.000);

\draw[Path,Blue] (7.000, 3.000) -- (8.000, 2.000);
\draw[Path,Blue] (8.000, 2.000) -- (9.000, 3.000);
\draw[Path,Blue] (9.000, 3.000) -- (10.000, 2.000);
\draw[Path,Blue] (10.000, 2.000) -- (11.000, 1.000);
\draw[DPoint,Blue] (6.000, 2.000) circle;
\draw[DPoint,Blue] (7.000, 3.000) circle;
\draw[DPoint,Blue] (8.000, 2.000) circle;
\draw[DPoint,Blue] (9.000, 3.000) circle;
\draw[DPoint,Blue] (10.000, 2.000) circle;
\draw[Path] (11.000, 1.000) -- (12.000, 2.000);
\draw[DPoint,Blue] (11.000, 1.000) circle;
\draw[Path] (12.000, 2.000) -- (13.000, 1.000);
\draw[Path] (13.000, 1.000) -- (14.000, 0.000);
\draw[Path] (14.000, 0.000) -- (15.000, 1.000);

\draw[Path] (15.000, 1.000) -- (16.000, 0.000);
\draw[DPoint] (12.000, 2.000) circle;
\draw[DPoint] (13.000, 1.000) circle;
\draw[DPoint] (14.000, 0.000) circle;
\draw[DPoint] (15.000, 1.000) circle;

\draw[DPoint] (16.000, 0.000) circle;
\end{tikzpicture}
} &
$\longrightarrow$
&
\scalebox{0.3}{
\begin{tikzpicture}
\draw[Path] (0.000, 0.000) -- (1.000, 1.000);
\draw[Path] (1.000, 1.000) -- (2.000, 2.000);
\draw[Path] (2.000, 2.000) -- (3.000, 1.000);
\draw[Path] (3.000, 1.000) -- (4.000, 2.000);
\draw[DPoint] (0.000, 0.000) circle;
\draw[DPoint] (1.000, 1.000) circle;
\draw[DPoint] (2.000, 2.000) circle;
\draw[DPoint] (3.000, 1.000) circle;
\draw[Path,Blue] (4.000, 2.000) -- (5.000, 3.000);
\draw[Path,Blue] (5.000, 3.000) -- (6.000, 4.000);

\draw[Path,Blue] (6.000, 4.000) -- (7.000, 3.000);
\draw[Path,Blue] (7.000, 3.000) -- (8.000, 4.000);
\draw[Path,Blue] (8.000, 4.000) -- (9.000, 3.000);
\draw[Path,Blue] (9.000, 3.000) -- (10.000, 2.000);
\draw[DPoint,Blue] (4.000, 2.000) circle;
\draw[DPoint,Blue] (5.000, 3.000) circle;
\draw[DPoint,Blue] (6.000, 4.000) circle;

\draw[DPoint,Blue] (7.000, 3.000) circle;
\draw[DPoint,Blue] (8.000, 4.000) circle;
\draw[DPoint,Blue] (9.000, 3.000) circle;
\draw[StrongPath,Red] (10.000, 2.000) -- (11.000, 1.000);
\draw[DPoint,Red] (10.000, 2.000) circle;
\draw[Path] (11.000, 1.000) -- (12.000, 2.000);
\draw[DPoint,Red] (11.000, 1.000) circle;
\draw[Path] (12.000, 2.000) -- (13.000, 1.000);
\draw[Path] (13.000, 1.000) -- (14.000, 0.000);
\draw[Path] (14.000, 0.000) -- (15.000, 1.000);

\draw[Path] (15.000, 1.000) -- (16.000, 0.000);
\draw[DPoint] (12.000, 2.000) circle;
\draw[DPoint] (13.000, 1.000) circle;
\draw[DPoint] (14.000, 0.000) circle;
\draw[DPoint] (15.000, 1.000) circle;

\draw[DPoint] (16.000, 0.000) circle;
\end{tikzpicture}
}
\end{tabular}}
\end{tabular}

\caption{Right rotation on binary trees and Dyck paths.}

\label{fig:tree-right-rotation}

\end{figure}

The Tamari lattice can also be described in terms of Dyck paths. A Dyck path of size $n$ is a lattice path from the origin $(0,0)$ to the point $(2n,0)$ made from a sequence of up steps $(1,1)$ and down steps $(1,-1)$ such that the path stays above the line $y=0$. These objects are counted by the Catalan numbers as well as binary trees. A simple bijection can be made between the two sets by considering the binary recursive structure of a Dyck path. Indeed, let $D$ be a Dyck path and $u$ the last up step of $D$ starting at $y=0$ (if $D$ never touches $y=0$ then $u$ is the first step of $D$). Then $D$ is made of a first Dyck path $D_1$, then the up step $u$ followed by a second Dyck path $D_2$ and a down step. The image of $D$ is then the binary tree $T$ made from $T_1$ and $T_2$, the images of respectively $D_1$ and $D_2$, see Figure \ref{fig:dyck-tree} for an example. Following this bijection, the right rotation easily translates in terms of Dyck path: it consists in switching a down step $d$ with the shortest Dyck path starting right after $d$ as illustrated on Figure \ref{fig:tree-right-rotation}.

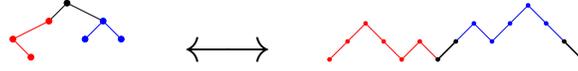
\begin{figure}[ht]
\centering
\scalebox{0.6}{




\begin{tabular}{ccccc}
\scalebox{0.8}{
\begin{tikzpicture}
\node (N0) at (1.750, 0.000){};
\node (N00) at (1.250, -0.500){};
\node (N000) at (0.250, -1.000){};
\node (N0001) at (0.750, -1.500){};
\draw (N0.center) -- (N00.center);
\draw[Point] (N0) circle;
\draw[Point, Red] (N0001) circle;
\draw[Red] (N000.center) -- (N0001.center);
\draw[Point, Red] (N000) circle;
\draw[Red] (N00.center) -- (N000.center);
\draw[Point, Red] (N00) circle;
\node (N01) at (2.750, -0.500){};
\draw (N0.center) -- (N01.center);
\node (N010) at (2.250, -1.000){};
\draw[Point, Blue] (N010) circle;
\node (N011) at (3.250, -1.000){};
\draw[Point, Blue] (N011) circle;
\draw[Blue] (N01.center) -- (N010.center);
\draw[Blue] (N01.center) -- (N011.center);
\draw[Point, Blue] (N01) circle;
\end{tikzpicture}
}
&
\hspace{0.5cm}
&
\scalebox{3}{$\longleftrightarrow$}
&
\hspace{0.5cm}
&
\scalebox{0.4}{
\begin{tikzpicture}
\draw[Path,Red] (0.000, 0.000) -- (1.000, 1.000);
\draw[Path,Red] (1.000, 1.000) -- (2.000, 2.000);
\draw[Path,Red] (2.000, 2.000) -- (3.000, 1.000);
\draw[Path,Red] (3.000, 1.000) -- (4.000, 0.000);
\draw[Path,Red] (4.000, 0.000) -- (5.000, 1.000);
\draw[Path,Red] (5.000, 1.000) -- (6.000, 0.000);
\draw[DPoint,Red] (0.000, 0.000) circle;
\draw[DPoint,Red] (1.000, 1.000) circle;
\draw[DPoint,Red] (2.000, 2.000) circle;
\draw[DPoint,Red] (3.000, 1.000) circle;
\draw[DPoint,Red] (4.000, 0.000) circle;
\draw[DPoint,Red] (5.000, 1.000) circle;
\draw[StrongPath] (6.000, 0.000) -- (7.000, 1.000);
\draw[DPoint] (6.000, 0.000) circle;
\draw[Path,Blue] (7.000, 1.000) -- (8.000, 2.000);
\draw[DPoint] (7.000, 1.000) circle;
\draw[Path,Blue] (8.000, 2.000) -- (9.000, 1.000);
\draw[Path,Blue] (9.000, 1.000) -- (10.000, 2.000);
\draw[Path,Blue] (10.000, 2.000) -- (11.000, 3.000);
\draw[Path,Blue] (11.000, 3.000) -- (12.000, 2.000);
\draw[Path,Blue] (12.000, 2.000) -- (13.000, 1.000);
\draw[DPoint,Blue] (8.000, 2.000) circle;
\draw[DPoint,Blue] (9.000, 1.000) circle;
\draw[DPoint,Blue] (10.000, 2.000) circle;
\draw[DPoint,Blue] (11.000, 3.000) circle;
\draw[DPoint,Blue] (12.000, 2.000) circle;
\draw[StrongPath] (13.000, 1.000) -- (14.000, 0.000);
\draw[DPoint] (13.000, 1.000) circle;
\draw[DPoint] (14.000, 0.000) circle;
\end{tikzpicture}
}
\end{tabular}}
    \caption{Bijection between Dyck paths and binary trees.}
    
    \label{fig:dyck-tree}
    
\end{figure}

\subsection{Construction of interval-posets}

We only give here a very short summary of the construction process. For a more detailed description, refer to \cite{Me_Tamari} where those objects were introduced.

A \emph{binary search tree} is a labelled binary tree where if a node is labelled $k$, all nodes of its left subtree have labels smaller than or equal to $k$, and all nodes on its right subtree have labels greater than $k$. There is a unique way to label a binary tree of size $n$ with labels in $\{1, \dots, n\}$ such that the result is a binary search tree. Such a labelled tree can be seen as a poset: $a$ precedes $b$ ($a \trprec b$) if $b$ is an ancestor of $a$. The linear extensions of the binary search tree labelled with labels in $1, \dots, n$ are permutations and form an interval of the weak order called the \emph{sylvester class} of the tree. The details of this construction can be found in \cite{PBT2}.

From a binary search tree $T$, one can construct bijectively two labelled forests of planar trees: the final forest $\dec(T)$ and the initial forest $\inc(T)$. The final forest is obtained by keeping only the decreasing relations of $T$. We write $b \trprec_{\dec(T)} a$ if and only if $b > a$ and $b \trprec_T a$, in other words if and only if $b$ is in the right subtree of $a$. Symmetrically, $\inc(T)$ is formed by the increasing relations of $T$, see the example on Figure~\ref{fig:example-forest}.

\begin{figure}[ht]
\centering

\begin{tabular}{|c|c|c|}
  \hline	
  Tree $T$ & $\inc(T)$ & $\dec(T)$ \\
  \hline
  \scalebox{0.6}{
    \begin{tikzpicture}
      \node(N1) at (-2,-1){1};
      \node(N2) at (-1.5,-3){2};
      \node(N3) at (-1,-2){3};
      \node(N4) at (-0.5,-3){4};
      \node(N5) at (0,0){5};
      \node(N6) at (1,-2){6};
      \node(N7) at (2,-1){7};
      \node(N8) at (2.5,-3){8};
      \node(N9) at (3,-2){9};
      \node(N10) at (3.5,-3){10};
      
      \draw (N1) -- (N5);
      \draw (N7) -- (N5);
      \draw (N3) -- (N1);
      \draw (N6) -- (N7);
      \draw (N9) -- (N7);
      \draw (N2) -- (N3);
      \draw (N4) -- (N3);
      \draw (N8) -- (N9);
      \draw (N10) -- (N9);
    \end{tikzpicture}
  }
  &
  \scalebox{0.6}{
    \begin{tikzpicture}
    \node(T1) at (0,0) {1};
    \node(T2) at (0,-1) {2};
    \node(T3) at (1,-1) {3};
    \node(T4) at (1,-2) {4};
    \node(T5) at (2,0) {5};
    \node(T6) at (3,0) {6};
    \node(T7) at (4,0) {7};
    \node(T8) at (3,-1) {8};
    \node(T9) at (4,-1) {9};
    \node(T10) at (4,-2) {10};
    \draw[color=blue] (T1) -- (T5);
    \draw[color=blue] (T2) -- (T3);
    \draw[color=blue] (T3) -- (T5);
    \draw[color=blue] (T4) -- (T5);
    \draw[color=blue] (T6) -- (T7);
    \draw[color=blue] (T8) -- (T9);
    \end{tikzpicture}
  }
  &
  \scalebox{0.6}{
    \begin{tikzpicture}
    \node(T2) at (0,-1) {2};
    \node(T4) at (1,-2) {4};
    \node(T3) at (1,-1) {3};
    \node(T1) at (1,0) {1};
    \node(T6) at (2,-1) {6};
    \node(T8) at (2,-2) {8};
    \node(T5) at (3,0) {5};
    \node(T7) at (3,-1) {7};
    \node(T9) at (3,-2) {9};
    \node(T10) at (3,-3) {10};
    \draw[color=red] (T2) -- (T1);
    \draw[color=red] (T3) -- (T1);
    \draw[color=red] (T4) -- (T3);
    \draw[color=red] (T6) -- (T5);
    \draw[color=red] (T7) -- (T5);
    \draw[color=red] (T8) -- (T7);
    \draw[color=red] (T9) -- (T7);
    \draw[color=red] (T10) -- (T9);
    \end{tikzpicture}
  }
  
  \\
  \hline
  
\end{tabular}
\caption{Initial and final forests of a binary tree. As a convention, interval-posets are always written both horizontally and vertically: $x \trprec y$ if there is a path between $x$ and $y$ and if $x$ is placed left of or below $y$. Increasing relations will be written in blue and in a more horizontal way from left to right whereas decreasing relations are in red and vertical from bottom to top.}
\label{fig:example-forest}
\end{figure}
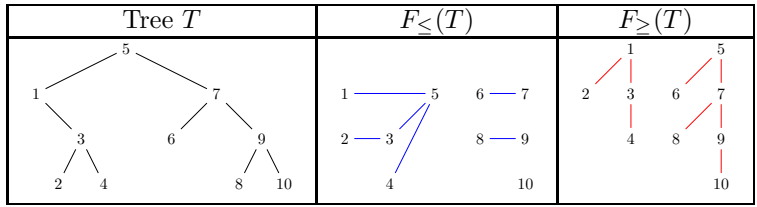

The linear extensions of $\inc(T)$ and $\dec(T)$ are respectively initial and final intervals of the weak order. The maximal (resp. minimal) permutation of the interval is the maximal (resp. minimal) permutation of the sylvester class of $T$. To construct the interval-poset of an interval $[T_1,T_2]$ we take all relations of both $\dec(T_1)$ and $\inc(T_2)$. If $\alpha$ is the minimal permutation of the sylvester class of $T_1$ and $\omega$ is the maximal permutation of the sylvester class of $T_2$, then the linear extensions of the interval posets are exactly the permutations $\mu$ satisfying $\alpha \leq \mu \leq \omega$. We proved in \cite{Me_Tamari} that these posets are exactly the ones satisfying the following property : for $a$ and $c$ labels of the poset such that $a <c$, $a \trprec c$ implies $b \trprec c$ for all $a < b < c$, and $c \trprec a$ implies $b \trprec a$ for all $a < b < c$.
Those posets are in bijection with intervals of the Tamari lattice. From a poset $I$, we can recover the poset $\inc(I)$ formed by increasing relations of $I$ which is in bijection with a binary tree $T_1$, and the poset $\dec(I)$ formed by decreasing relations of $I$ which is in bijection with a binary tree $T_2 \geq T_1$. The construction is illustrated on Figure \ref{fig:interval-poset-construction}.

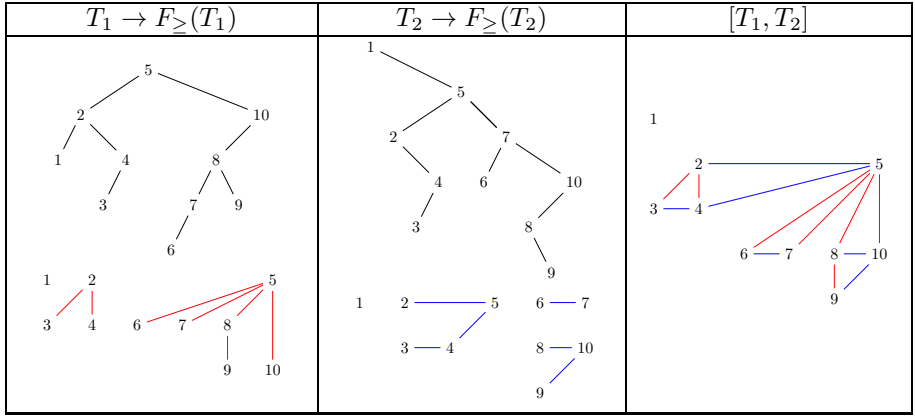
\begin{figure}[ht]
\centering
\begin{tabular}{|c|c|c|}
\hline
$T_1 \rightarrow \dec(T_1)$ & $T_2 \rightarrow \dec(T_2)$ & $[T_1, T_2]$ \\
\hline
\begin{tabular}{c}
\scalebox{0.6}{
\begin{tikzpicture}

	\node(N1) at (0,-2){1};
	\node(N2) at (0.5,-1){2};
	\node(N4) at (1.5,-2){4};
	\node(N3) at (1,-3){3};
    \node(N5) at (2,0){5};
    \node(N10) at (4.5,-1){10};
    \node(N8) at (3.5,-2){8};
    \node(N7) at (3,-3){7};
    \node(N9) at (4,-3){9};
    \node(N6) at (2.5,-4){6};
    \draw (N2) -- (N5); 
    \draw (N10) -- (N5);
    \draw (N1) -- (N2);
    \draw (N4) -- (N2);
    \draw (N8) -- (N10);
    \draw (N3) -- (N4);
    \draw (N7) -- (N8);
    \draw (N9) -- (N8);
    \draw (N6) -- (N7);
  \end{tikzpicture}
  } \\
  \scalebox{0.6}{
  \begin{tikzpicture}
  \node(T1) at (0,0) {1};
  \node(T3) at (0,-1) {3};
  \node(T4) at (1,-1) {4};
  \node(T2) at (1,0) {2};
  \node(T6) at (2,-1) {6};
  \node(T7) at (3,-1) {7};
  \node(T9) at (4,-2) {9};
  \node(T8) at (4,-1) {8};
  \node(T5) at (5,0) {5};
  \node(T10) at (5,-2) {10};
  \draw[color=red] (T3) -- (T2);
  \draw[color=red] (T4) -- (T2);
  \draw[color=red] (T6) -- (T5);
  \draw[color=red] (T7) -- (T5);
  \draw[color=red] (T8) -- (T5);
  \draw[color=red] (T9) -- (T8);
  \draw[color=red] (T10) -- (T5);
  \end{tikzpicture}

  }
\end{tabular}
  &
  \begin{tabular}{c}
  \scalebox{0.6}{
\begin{tikzpicture}
  
  \node(N1) at (0,0){1};
    \node(N5) at (2,-1){5};
    \node(N2) at (0.5,-2){2};
    \node(N4) at (1.5,-3){4};
    \node(N3) at (1, -4){3};
    
    \node(N7) at (3,-2){7};
    \node(N6) at (2.5,-3){6};
    \node(N10) at (4.5,-3){10};
    \node(N8) at (3.5, -4){8};
    \node(N9) at (4, -5){9};
    
    \draw (N5) -- (N1);
    \draw (N2) -- (N5);
    \draw (N7) -- (N5);
    \draw (N7) -- (N5);
    \draw (N4) -- (N2);
    \draw (N6) -- (N7);
    \draw (N10) -- (N7);
    \draw (N3) -- (N4);
    \draw (N8) -- (N10);
    \draw (N9) -- (N8);
    
  \end{tikzpicture}} \\
  \scalebox{0.6}{
  \begin{tikzpicture}
  \node(T1) at (0,0) {1};
  \node(T2) at (1,0) {2};
  \node(T3) at (1,-1) {3};
  \node(T4) at (2,-1) {4};
  \node(T5) at (3,0) {5};
  \node(T6) at (4,0) {6};
  \node(T7) at (5,0) {7};
  \node(T8) at (4,-1) {8};
  \node(T9) at (4,-2) {9};
  \node(T10) at (5,-1) {10};
  \draw[color=blue] (T2) -- (T5);
  \draw[color=blue] (T3) -- (T4);
  \draw[color=blue] (T4) -- (T5);
  \draw[color=blue] (T6) -- (T7);
  \draw[color=blue] (T8) -- (T10);
  \draw[color=blue] (T9) -- (T10);
  \end{tikzpicture}}
  \end{tabular}
   & 
   \scalebox{0.6}{
   \begin{tikzpicture}[baseline=-2.5cm]
   \node(T1) at (0,0) {1};
   \node(T3) at (0,-2) {3};
   \node(T4) at (1,-2) {4};
   \node(T2) at (1,-1) {2};
   \node(T6) at (2,-3) {6};
   \node(T7) at (3,-3) {7};
   \node(T9) at (4,-4) {9};
   \node(T8) at (4,-3) {8};
   \node(T5) at (5,-1) {5};
   \node(T10) at (5,-3) {10};
   \draw[color=blue] (T2) -- (T5);
   \draw[color=blue] (T3) -- (T4);
   \draw[color=red] (T3) -- (T2);
   \draw[color=blue] (T4) -- (T5);
   \draw[color=red] (T4) -- (T2);
   \draw[color=blue] (T6) -- (T7);
   \draw[color=red] (T6) -- (T5);
   \draw[color=red] (T7) -- (T5);
   \draw[color=blue] (T8) -- (T10);
   \draw[color=red] (T8) -- (T5);
   \draw[color=blue] (T9) -- (T10);
   \draw[color=red] (T9) -- (T8);
   \draw[color=red] (T10) -- (T5);
   \end{tikzpicture}
   }
  \\ \hline
\end{tabular}
\caption{Construction of an interval-poset.}
\label{fig:interval-poset-construction}
\end{figure}

\section{Initial Rise}
\label{sec:initial-rise}

\subsection{Initial rise of interval-posets}

In \cite{mTamari}, the authors give a functional equation of the generating function of Tamari intervals depending on two statistics. The statistics are given in terms of Dyck paths. The first one is the number of contacts between the lower Dyck path of the interval and the $x$-axis. On an interval-poset $I$, this statistic corresponds to $\Itrees(I)$, \emph{i.e.}, the number of components of the final forest of $I$ \cite{Me_Tamari}. In all three different methods used in \cite{Chap, mTamari, Me_Tamari} to generate intervals, this statistics is crucial to obtain the functional equation. Following the notation of \cite{mTamari}, we call it the \emph{catalytic} parameter. The authors of \cite{mTamari} also introduce a second (non essential) statistic: the \emph{initial rise} of an interval is the initial rise of its upper path, \emph{i.e.}, the number of initial up steps of the path. By simply running through the previously described bijection between Dyck paths and binary trees and following the construction process of an interval-poset, one can read this statistic directly on the interval-poset. It is the number of initial vertices $1,2,\dots,k$ such that there is no relation $k-1 \trprec k$, see Figure \ref{fig:initial-rise}. We call it the initial rise of the interval-poset and write $\Iir(I)$. If $\Iir(I) = k$ it means that $k$ is either the first vertex with $k \trprec k+1$, or $k = \Isize(I)$ and $I$ has no increasing relations.

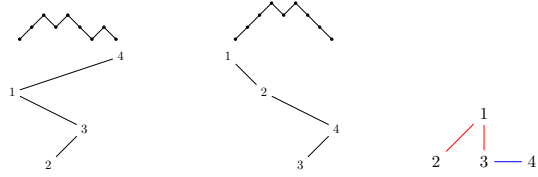
\begin{figure}[ht]
\centering
\scalebox{0.8}{
\begin{tabular}{ccccc}
\scalebox{0.2}{


\begin{tikzpicture}
\draw[Path] (0.000, 0.000) -- (1.000, 1.000);
\draw[Path] (1.000, 1.000) -- (2.000, 2.000);
\draw[Path] (2.000, 2.000) -- (3.000, 1.000);
\draw[Path] (3.000, 1.000) -- (4.000, 2.000);
\draw[Path] (4.000, 2.000) -- (5.000, 1.000);
\draw[Path] (5.000, 1.000) -- (6.000, 0.000);
\draw[Path] (6.000, 0.000) -- (7.000, 1.000);
\draw[Path] (7.000, 1.000) -- (8.000, 0.000);
\draw[DPoint] (0.000, 0.000) circle;
\draw[DPoint] (1.000, 1.000) circle;
\draw[DPoint] (2.000, 2.000) circle;
\draw[DPoint] (3.000, 1.000) circle;
\draw[DPoint] (4.000, 2.000) circle;
\draw[DPoint] (5.000, 1.000) circle;
\draw[DPoint] (6.000, 0.000) circle;
\draw[DPoint] (7.000, 1.000) circle;
\draw[DPoint] (8.000, 0.000) circle;
\end{tikzpicture}} & 
\hspace{1cm}
\scalebox{0.2}{


\begin{tikzpicture}
\draw[Path] (0.000, 0.000) -- (1.000, 1.000);
\draw[Path] (1.000, 1.000) -- (2.000, 2.000);
\draw[Path] (2.000, 2.000) -- (3.000, 3.000);
\draw[Path] (3.000, 3.000) -- (4.000, 2.000);
\draw[Path] (4.000, 2.000) -- (5.000, 3.000);
\draw[Path] (5.000, 3.000) -- (6.000, 2.000);
\draw[Path] (6.000, 2.000) -- (7.000, 1.000);
\draw[Path] (7.000, 1.000) -- (8.000, 0.000);
\draw[DPoint] (0.000, 0.000) circle;
\draw[DPoint] (1.000, 1.000) circle;
\draw[DPoint] (2.000, 2.000) circle;
\draw[DPoint] (3.000, 3.000) circle;
\draw[DPoint] (4.000, 2.000) circle;
\draw[DPoint] (5.000, 3.000) circle;
\draw[DPoint] (6.000, 2.000) circle;
\draw[DPoint] (7.000, 1.000) circle;
\draw[DPoint] (8.000, 0.000) circle;
\end{tikzpicture}} \\

\scalebox{0.6}{


\begin{tikzpicture}
\node (N0) at (3.500, 0.000){4};
\node (N00) at (0.500, -1.000){1};
\node (N001) at (2.500, -2.000){3};
\node (N0010) at (1.500, -3.000){2};
\draw (N001) -- (N0010);
\draw (N00) -- (N001);
\draw (N0) -- (N00);
\end{tikzpicture}} & 
\hspace{1cm}
\scalebox{0.6}{


\begin{tikzpicture}
\node (N0) at (0.500, 0.000){1};
\node (N01) at (1.500, -1.000){2};
\node (N011) at (3.500, -2.000){4};
\node (N0110) at (2.500, -3.000){3};
\draw (N011) -- (N0110);
\draw (N01) -- (N011);
\draw (N0) -- (N01);
\end{tikzpicture}} &
\hspace{0.5cm} &
\scalebox{0.8}{
\begin{tikzpicture}\node(T2) at (0,-1) {2};\node(T3) at (1,-1) {3};\node(T1) at (1,0) {1};\node(T4) at (2,-1) {4};\draw[color=red] (T2) -- (T1);\draw[color=blue] (T3) -- (T4);\draw[color=red] (T3) -- (T1);\end{tikzpicture}}
\end{tabular}
}
\caption{Initial rise and contacts of Tamari intervals. In this example, the lower path has 2 non initial contacts with the $x$-axis. They correspond to the 2 vertices (1 and 4) on the left border of the smaller binary tree and to the 2 components of the final forest of the interval-poset. The initial rise of the upper path is 3 because it starts with 3 consecutive up steps. They corresponds to nodes labelled 1,2, and 3 on the binary tree. By construction, those nodes have no left subtree which means that the vertices 1,2 and 3 in the interval-poset have no relation $k-1 \trprec k$.}
\label{fig:initial-rise}
\end{figure}

Let $\Phi(y;x,z)$ be the generating function of intervals fo Tamari where $y,x,$ and $z$ respectively count $\Isize(I)$, $\Itrees(I)$, and $\Iir(I)$. 

\begin{align}
\label{eq:initial-rise-generating-function}
\Phi(y;x,z) &= \sum_{I} y^{\Isize(I)}x^{\Itrees(I)}z^{\Iir(I)}, \\
&= 1 + y\,xz + y^2\,(x^2 z^2 +x^2 z + xz^2)\\
&+ y^3\,(x^3 z^3 + 2 x^3 z^2 + 2 x^3 z + 2 x^2 z^3 + 2x^2 z^2 + x^2 z + 2 x z^3 + x z^2) + \cdots
\end{align}

In \cite{mTamari}, it has been proved that $\Phi$ satisfies the following functional equation,

\begin{equation}
\label{eq:initial-rise-functional-equation}
\Phi(y;x,z) = 1 + xyz \Phi(y;x,1) \frac{x \Phi(y;x,z) - \Phi(y;1,z)}{x - 1}.
\end{equation}
By solving this equation, the authors found that the joint distribution of $\Itrees(I)$ and $\Iir(I)$ is symmetric, \emph{i.e.}, that $\Phi(y;x,z) = \Phi(y;z,x)$. The question of a combinatorial proof was left open. The aim of this section is to give such a proof by describing the explicit bijection that exchanges the two statistics on interval-posets.

\subsection{Two decompositions of intervals}

The main idea of the bijection is that an interval-poset can be decomposed in two different ways into two smaller interval-posets. One way is given by the composition operation described in \cite{Me_Tamari}.

\begin{Proposition}
An interval-poset $I$ of size $n$ is fully determined by a triplet $(I_1, I_2, r)$ where $I_1$ and $I_2$ are two interval-posets with $\Isize(I_1) + \Isize(I_2) + 1 = \Isize(I)$ and $r$ is an integer such that $0 \leq r \leq \Itrees(I_2)$. We call this decomposition the \emph{lower contacts decomposition} of the interval and we write $I = LC(I_{1}, I_{2}, r)$.
\end{Proposition}

This proposition is a direct consequence of \cite[Prop. 3.7]{Me_Tamari}. Let $(I_1, I_2, r)$ be such a triplet and $I_2$ be such that $\Itrees(I_2) = s$ with $x_1 \leq x_2 \leq \dots \leq x_s$ the roots of $\dec(I_2)$. Then $I$ is the shifted concatenation of $I_1$, a new vertex $k$, and $I_2$ with $y \trprec k$ for all $y \in I_1$ and $x_i \trprec k$ for $1 \leq i \leq r$. Conversely, if $I$ is an interval-poset, its root $k$ is the vertex with maximal label satisfying $i \trprec k$ for all $i < k$. Then $I_1$ is the subposet formed by vertices $i < k$ and $I_2$ is the subposet formed by vertices $j > k$. Finally, $r$ is the number of children of $k$ in $\dec(I)$. See the example bellow.

\begin{equation}
\begin{array}{ccccc}
\label{eq:lower-contact-dec}
\scalebox{0.6}{

\begin{tikzpicture}\node(T1) at (0,0) {1};\node(T3) at (1,-1) {3};\node(T2) at (1,0) {2};\node(T5) at (2,-2) {5};\node(T6) at (3,-2) {6};\node(T4) at (3,0) {4};\node(T7) at (4,-2) {7};\node(T8) at (4,-3) {8};\draw[color=blue] (T1) -- (T2);\draw[color=blue] (T2) -- (T4);\draw[color=blue] (T3) -- (T4);\draw[color=red] (T3) -- (T2);\draw[color=red] (T5) -- (T4);\draw[color=blue] (T6) -- (T7);\draw[color=red] (T6) -- (T4);\draw[color=red] (T8) -- (T7);\end{tikzpicture}} & 
\hspace{1cm} &
= &
\hspace{1cm} &
LC
\left( \scalebox{0.6}{

\begin{tikzpicture}\node(T1) at (0,0) {1};\node(T2) at (1,0) {2};\node(T3) at (1,-1) {3};\draw[color=blue] (T1) -- (T2);\draw[color=red] (T3) -- (T2);\end{tikzpicture}

},
\scalebox{0.6}{


\begin{tikzpicture}
  \node(T1) at (-1,-1) {1};
  \node(T2) at (0,-1) {2};
  \node(T3) at (1,-1) {3};
  \node(T4) at (1,-2) {4};
  \draw[color=blue] (T2) -- (T3);
  \draw[color=red] (T4) -- (T3);
\end{tikzpicture}

},
2 \right)
\end{array}
\end{equation}



We now give a new way to decompose the interval.

\begin{Proposition}
An interval-poset $I$ of size $n$ is fully determined by a triplet $(I_1, I_2, r)$ where $I_1$ and $I_2$ are two interval-posets with $\Isize(I_1) + \Isize(I_2) + 1 = \Isize(I)$ and $r$ is an integer such that $0 \leq r \leq \Iir(I_2)$. We call this decomposition the \emph{initial rise decompostion} and we write $I = IR(I_{1}, I_{2}, r)$.
\end{Proposition}

This decomposition has not been described before. It comes from a new composition operation between interval-posets that we call the \emph{initial rise} composition. It is described in two steps. First, let $I_2$ be an interval-poset and $r$ such that $0 \leq r \leq \Iir(I_2)$, and let us insert a new vertex into $I_2$ to obtain an interval-poset $I_2'$. The label of the new vertex is $k = \Iir(I_2) - r + 1$ and the labels of the vertices of $I_2$ are shifted accordingly (the smaller ones are unchanged and the larger ones are shifted by 1). The increasing relations of $I_2$ are left unchanged and an extra relation $k \trprec k+1$ is added if $k+1 \leq \Isize(I_2)$. The decreasing relations are replaced such that the number of children of each former vertex of $I_2$ is the same in $\dec(I_2')$ as it was in $\dec(I_2)$ (the condition on the decreasing relations of the interval-posets implies that there is only one way of satisfying this condition). This insertion process is illustrated by Figure \ref{fig:initial-rise-insertion}.

\begin{figure}[ht]
\centering
\begin{tabular}{ccc}
\scalebox{0.6}{
\begin{tikzpicture}\node(T2) at (0,-1) {2};\node(T4) at (1,-3) {4};\node(T3) at (1,-2) {3};\node(T1) at (2,0) {1};\node(T5) at (2,-1) {5};\node(T6) at (2,-4) {6};\draw[color=blue] (T2) -- (T5);\draw[color=red] (T2) -- (T1);\draw[color=blue] (T3) -- (T5);\draw[color=red] (T3) -- (T1);\draw[color=blue] (T4) -- (T5);\draw[color=red] (T4) -- (T3);\draw[color=red] (T5) -- (T1);\draw[color=red] (T6) -- (T5);
\end{tikzpicture}}
&
\scalebox{0.6}{
\begin{tikzpicture}\node(T2) at (0,-1) {2};\node(T3) at (1,-2) {3};\node(T5) at (2,-3) {5};\node(T4) at (2,-2) {4};\node(T1) at (2,0) {1};\node(T6) at (3,-1) {6};\node(T7) at (3,-4) {7};\draw[color=blue] (T2) -- (T6);\draw[color=blue] (T3) -- (T4);\draw[color=blue] (T4) -- (T6);\draw[color=blue] (T5) -- (T6);
\end{tikzpicture}}
&
\scalebox{0.6}{
\begin{tikzpicture}\node(T2) at (0,-1) {2};\node(T3) at (1,-2) {3};\node(T5) at (2,-3) {5};\node(T4) at (2,-2) {4};\node(T1) at (2,0) {1};\node(T6) at (3,-1) {6};\node(T7) at (3,-4) {7};\draw[color=blue] (T2) -- (T6);\draw[color=red] (T2) -- (T1);\draw[color=blue] (T3) -- (T4);\draw[color=red] (T3) -- (T1);\draw[color=blue] (T4) -- (T6);\draw[color=red] (T4) -- (T1);\draw[color=blue] (T5) -- (T6);\draw[color=red] (T5) -- (T4);\draw[color=red] (T7) -- (T6);
\end{tikzpicture}} \\
$\Iir(I_2) = 4$ &
\parbox{4.5cm}{We do the insertion for $r=2$, so the inserted vertex is 3.} &
\parbox{4.5cm}{We replace the decreasing relations so that every former vertex of $I_2$ has the same number of children.}
\end{tabular}
\caption{Insertion into an interval-poset for the initial rise composition with $r = 2$.}
\label{fig:initial-rise-insertion}
\end{figure}
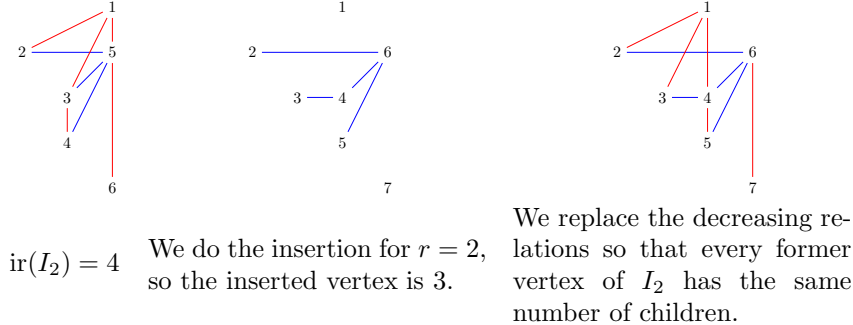

The second step of the composition consists in merging $I_1$ and $I_2'$. Let $a = \Iir(I_1)$. Insert $I_2'$ right after $a$, which means shift the vertices of $I_2'$ by $a$ and the vertices of $I_1$ bigger than $a$ by $\Isize(I_2')$. Then add decreasing relations $j \trprec a$ for all $j$ of $I_2'$. Finally, if $a \neq \Isize(I_1)$ then there was a relation $a \trprec a+1$ in $I_1$ which is now $a \trprec a+1+\Isize(I_2') = b$. We then add all increasing relations $j \trprec b$ for all $j$ in $I_{2}'$.

\begin{figure}[ht]
\centering
\scalebox{0.8}{\begin{tabular}{ccccc}
\scalebox{0.8}{
\begin{tikzpicture}\node(T2) at (0,-1) {2};\node(T1) at (0,0) {1};\node(T3) at (1,0) {3};\node(T4) at (1,-2) {4};\draw[color=blue] (T1) -- (T3);\draw[color=blue] (T2) -- (T3);\draw[color=red] (T2) -- (T1);\draw[color=red] (T4) -- (T3);
\end{tikzpicture}}
&
\hspace{0.8cm}
&
\scalebox{0.8}{
\begin{tikzpicture}\node(T1) at (0,0) {1};\node(T2) at (1,0) {2};\node(T3) at (1,-1) {3};\draw[color=blue] (T1) -- (T2);\draw[color=red] (T3) -- (T2);
\end{tikzpicture}}
&
\hspace{0.8cm}
&
\scalebox{0.8}{
\begin{tikzpicture}\node(T3) at (0,-2) {3};\node(T5) at (1,-3) {5};\node(T4) at (1,-2) {4};\node(T2) at (1,-1) {2};\node(T1) at (1,0) {1};\node(T6) at (2,0) {6};\node(T7) at (2,-4) {7};\draw[color=blue] (T1) -- (T6);\draw[color=blue] (T2) -- (T6);\draw[color=red] (T2) -- (T1);\draw[color=blue] (T3) -- (T4);\draw[color=red] (T3) -- (T2);\draw[color=blue] (T4) -- (T6);\draw[color=red] (T4) -- (T2);\draw[color=blue] (T5) -- (T6);\draw[color=red] (T5) -- (T4);\draw[color=red] (T7) -- (T6);
\end{tikzpicture}} \\
$I_1$ & 
\hspace{0.8cm} &
$I_2'$ & 
\hspace{0.8cm} &
$I$
\end{tabular}}
\caption{Construction of $I$ from $I_1$ and $I_2'$.}
\end{figure}

Note that this process can be reversed: an interval $I$ is uniquely decomposed into $I_1$ and $I_2$. The vertex $a$ of $I_1$ that we used to merge $I_1$ and $I_2'$ is the vertex of $I$ with maximal label such that
\begin{itemize}
\item  $a \leq \Iir(I)$
\item $j \trprec a$ for all $a \leq j \leq \Iir(I)$
\item and if $\Iir(I) \neq \Isize(I)$ then $a \trprec \Iir(I) +1$.
\end{itemize}
The interval-poset $I_2'$ is then the subposet formed by vertices $a < j < \Iir(I) + 1$. Then, the inserted vertex in $I_2'$ is always its initial rise by construction. We remove it by reversing the process of Figure \ref{fig:initial-rise-insertion}. We can now define a recursive bijection.

\begin{Definition}
Let $I$ be an interval-poset and $I = LC(I_1,I_2,r)$ its lower contacts decomposition. Then $\beta(I)$ is recursively defined by
\begin{itemize}
\item $\beta(\emptyset) = \emptyset$ 
\item $\beta(I) = IR(\beta(I_1), \beta(I_2), r)$.
\end{itemize}
\end{Definition}

\begin{Proposition}
Let $I$ be an interval-poset such that $\Itrees(I) = a$ and $\Iir(I) = b$, then $\Itrees(\beta(I)) = b$ and $\Iir(\beta(I)) = a$.
\end{Proposition}

\begin{proof}
Let $I$ be an interval-poset and $(I_1,I_2,r)$ its lower contacts decomposition, and let $J_1 = \beta(I_1)$,  $J_2 = \beta(I_2)$, and $J = \beta(I)$ the initial-rise composition of $(J_1,J_2,r)$. We have by construction  $\Itrees(I) = \Itrees(I_1) + 1 + \Itrees(I_2) - r$, and $\Iir(J) = \Iir(J_1) + \Iir(J_2') = \Iir(J_1) + \Iir(J_2) + 1 - r$ which proves the first part of the proposition. For the second statistic, we have to consider two cases. Let us first suppose that $I_1$ is not empty. Then $\Iir(I) = \Iir(I_1)$ and also $\Itrees(J) = \Itrees(J_1)$. Now, if $I_1$ is empty, then $\Iir(I) = \Iir(I_2) + 1$ and $\Itrees(J) = \Itrees(J_2') = \Itrees(J_2) + 1$.
\end{proof}

\vspace{-0.1cm}

\subsection{Example}

\vspace{-0.1cm}

We now detail the computation of the image of an interval poset by $\beta$.

Let $I = LC(I_{1}, I_{2}, r)$ be the interval poset,

\begin{center}

\begin{equation}
I = \scalebox{0.7}{

\begin{tikzpicture}
  \node(T1) at (0,0) {1};
  \node(T3) at (1,-1) {3};
  \node(T2) at (1,0) {2};
  \node(T5) at (2,-2) {5};
  \node(T4) at (2,0) {4};
  \node(T6) at (3,-2) {6};
  
  \draw[color=blue] (T1) -- (T2);
  \draw[color=blue] (T2) -- (T4);
  \draw[color=blue] (T3) -- (T4);
  \draw[color=red] (T3) -- (T2);
  \draw[color=blue] (T5) -- (T6);
  \draw[color=red] (T5) -- (T4);
\end{tikzpicture}} = LC
\left( \scalebox{0.7}{\begin{tikzpicture}
  \node(T1) at (0,0) {1};
  \node(T2) at (1,0) {2};
  \node(T3) at (1,-1) {3};

  \draw[color=blue] (T1) -- (T2);
  \draw[color=red] (T3) -- (T2);
\end{tikzpicture}
},
\scalebox{0.7}{

\begin{tikzpicture}
  \node(T1) at (0,0) {1};
  \node(T2) at (1,0) {2};

  \draw[color=blue] (T1) -- (T2);
\end{tikzpicture}
},
1 \right).
\end{equation}

\end{center}

To compute $\beta(I)$, we first need to compute $\beta(I_{1})$ and
$\beta(I_{2})$. Let $I_{1} = LC(I_{1,1}, I_{1,2}, r_{1})$ be the
lower contact decomposition of $I_{1}$.

As $\beta(I_{1,1}) = I_{1,1}$ and $\beta(I_{1,2}) = I_{1,2}$, we just
need to use the initial rise composition in order to compute
$\beta(I_{1})$.

First, we compute $k = ir(I_{1,2}) - r_{1} + 1 = 1$, we add it to
$I_{1,2}$ and we shift the vertices which are greater than or equal to
$k$. As $k + 1 \leq \Isize(I_{1,2}')$, we add a relation $k \trprec
k+1$.

We can now merge $I_{1,1}$ and $I_{1,2}'$. To do this, we compute $a =
ir(I_{1,1}) = 1$ and we insert $I_{1,2}'$ right after $a$ and then we
add $j \trprec a$ for all $j$ in $I_{1,2}'$ which gives us:

\vspace{-0.2cm}

\begin{equation}
  \begin{array}{ccc}
    I_{1} = \scalebox{0.7}{\begin{tikzpicture}
  \node(T1) at (0,0) {1};
  \node(T2) at (1,0) {2};
  \node(T3) at (1,-1) {3};
  \draw[color=blue] (T1) -- (T2);
  \draw[color=red] (T3) -- (T2);
\end{tikzpicture}
} = LC
    \left(
    1,
    1,
    1
    \right)

    \hspace{2cm}

    &
    \Rightarrow
    &

    \beta(I_{1}) 
    = IR 
    \left(
    1,
    1,
    1
    \right)
    = \scalebox{0.7}{\begin{tikzpicture}
  \node(T2) at (0,-1) {2};
  \node(T1) at (1,0) {1};
  \node(T3) at (1,-1) {3};

  \draw[color=blue] (T2) -- (T3);
  \draw[color=red] (T2) -- (T1);
  \draw[color=red] (T3) -- (T1);
\end{tikzpicture}

}
  \end{array}
\end{equation}



The next step is to compute $\beta(I_{2})$. As $\Isize(I_{2}) = 2$, we
do not give the details of the computation.

\vspace{-0.2cm}

\begin{equation}
  \begin{array}{ccc}
    I_{2} =
    \scalebox{0.7}{\begin{tikzpicture}
  \node(T1) at (0,0) {1};
  \node(T2) at (1,0) {2};

  \draw[color=blue] (T1) -- (T2);
\end{tikzpicture}
  
} = LC
    \left(
    1,
    \emptyset,
    0
    \right)

    \hspace{2cm}

    &
    \Rightarrow
    &

    \beta(I_{2}) 
    = IR 
    \left(
    1,
    \emptyset,
    0
    \right)
    = \scalebox{0.7}{\begin{tikzpicture}
  \node(T1) at (0,0) {1};
  \node(T2) at (0,-1) {2};

  \draw[color=red] (T2) -- (T1);
\end{tikzpicture}
}
  \end{array}
\end{equation}





Now that we have computed $\beta(I_{1})$ and $\beta(I_{2})$, we can
compute $\beta(I)$. The first step is to compute $k = ir(I_{2}) - r +
1 = 2$. Now we insert $k$ into $I_{2}$ and shift the vertices
accordingly. The increasing relation of $I_{2}$ are left unchanged and
we add an extra relation $k \trprec k + 1$ as $k \leq
\Isize(I_{2})$. The decreasing relations are replaced such that each
former vertex of $I_{2}$ has the same number of decreasing relations.
In our case, 1 had only one decreasing relation and there is only one
way to add this relation in $I_{2}'$ which is $2 \trprec 1$, so that


\vspace{-0.2cm}

\begin{equation}
  I_{2}' =
  \scalebox{0.7}{\begin{tikzpicture}
  \node(T2) at (0,-1) {2};
  \node(T1) at (0,0) {1};
  \node(T3) at (1,-1) {3};

  \draw[color=blue] (T2) -- (T3);
  \draw[color=red] (T2) -- (T1);
\end{tikzpicture}


}
\end{equation}

The final step consists in merging $I_{1}$ and $I_{2}'$. We compute $a
= ir(I_{1}) = 2$. We insert $I_{2}'$ right after $a$ shifting the
labels accordingly. Then for all $j$ in $I_{2}'$, we add a decreasing
$j \trprec a$ and an increasing relation $j \trprec a + \Isize(I_{2}') +
1$.


\vspace{-0.2cm}

\begin{equation}
  \begin{array}{ccc}
    I = 
    LC \left(
    \scalebox{0.7}{
    \begin{tikzpicture}\node(T1) at (0,0) {1};\node(T2) at (1,0) {2};\node(T3) at (1,-1) {3};\draw[color=blue] (T1) -- (T2);\draw[color=red] (T3) -- (T2);\end{tikzpicture}
    }
    ,
    \scalebox{0.7}{
    \begin{tikzpicture}\node(T1) at (0,0) {1};\node(T2) at (1,0) {2};\draw[color=blue] (T1) -- (T2);\end{tikzpicture}
    }
    , 1\right)

    \hspace{2cm}

    &
    \Rightarrow
    &

    \beta(I) = 
    IR
    \left(
    \scalebox{0.7}{
    \begin{tikzpicture}\node(T2) at (0,-1) {2};\node(T1) at (1,0) {1};\node(T3) at (1,-1) {3};\draw[color=blue] (T2) -- (T3);\draw[color=red] (T2) -- (T1);\draw[color=red] (T3) -- (T1);\end{tikzpicture}
    }
    ,
    \scalebox{0.7}{
    \begin{tikzpicture}\node(T1) at (0,0) {1};\node(T2) at (0,-1) {2};\draw[color=red] (T2) -- (T1);\end{tikzpicture}
    }
    , 
    1\right)
    =
    \scalebox{0.6}{
    \begin{tikzpicture}\node(T4) at (0,-3) {4};\node(T3) at (0,-2) {3};\node(T5) at (1,-3) {5};\node(T2) at (1,-1) {2};\node(T1) at (2,0) {1};\node(T6) at (2,-1) {6};\draw[color=blue] (T2) -- (T6);\draw[color=red] (T2) -- (T1);\draw[color=blue] (T3) -- (T6);\draw[color=red] (T3) -- (T2);\draw[color=blue] (T4) -- (T5);\draw[color=red] (T4) -- (T3);\draw[color=blue] (T5) -- (T6);\draw[color=red] (T5) -- (T2);\draw[color=red] (T6) -- (T1);\end{tikzpicture}
    }

  \end{array}
\end{equation}

We easily check that $\Itrees(I) = \Iir(\beta(I)) = 4$ and that $\Iir(I) = \Itrees(\beta(I)) = 1$.

\vspace{-0.1cm}

\subsection{Other comments}

\vspace{-0.1cm}

It is also possible to decompose an interval-poset in terms of the initial rise composition. Indeed, an interval-poset $I$ appears in a unique initial rise composition of two interval-posets $I_1$ and $I_2$. These two intervals are not the bijective image of the two intervals obtained by a lower contact decomposition. As an example, the initial rise decomposition of the left interval poset of \eqref{eq:lower-contact-dec} gives two intervals of respective size 0 and 7, which cannot be the images of the ones obtained by the lower contact decomposition. However, automatic tests shows that if we recursively decompose an interval $I$ in terms of the initial rise and recompose it by applying the lower contact composition, we still obtain $\beta(I)$. It means that this bijection would actually be an involution on interval-posets. This result is surprising, it has been tested on interval-posets up to size 6 but we do not have yet a formal explanation. 

An interesting remark is that if we add the initial rise parameter to the functional equation we obtain in \cite{Me_Tamari} we do not retrieve the functional equation \eqref{eq:initial-rise-functional-equation} obtained by \cite{mTamari}. Indeed, we get
\begin{equation}
\Phi(y;x,z) = 1 + xyz \frac{x \Phi(y;x,z) - \Phi(y;1,z)}{x - 1} + xy \left( \Phi(y;x,z) - 1 \right)\frac{x \Phi(y;x,1) - \Phi(y;1,1)}{x - 1}.
\end{equation}
Both functional equations are true but they are not trivially equal from an algebraic point of view and call for further investigation.

To finish with, the symmetric distribution of the two statistics is also true for the $m$-Tamari lattices as shown in \cite{mTamari}. We are confident that the combinatorial proof we just gave for the classical case easily generalizes to $m$-Tamari using the results of \cite{Me_Tamari}. We intend to give the details of this construction is some future work.

\section{Flows}
\label{sec:flows}

\vspace{-0.1cm}

\subsection{Definition}

\vspace{-0.1cm}

Let $F$ be a forest of rooted ordered trees. We define a \emph{flow} on $F$ by attaching an \emph{input} $i \geq -1$ on each node of $F$ such that the \emph{outgoing rate} of each node is greater than or equal to 0. The outgoing rate of a node is the sum of the rates of its children nodes plus its own input. In particular, if a node has no children, its outgoing rate is its input. Inputs can be understood as \emph{sources} or \emph{leaks} of some fluid flowing from the nodes towards the root. The condition on the outgoing rate just expresses that the inner flow is never negative. One consequence is that a leak (\emph{i.e.}, an input of value $-1$) can never be placed on a leaf. An example of a flow is given in Figure \ref{fig:flow-example}. The sum of the outgoing rates of the roots is called the \emph{exit rate} of the flow. If the exit rate is 0, the flow is said to be \emph{closed}.

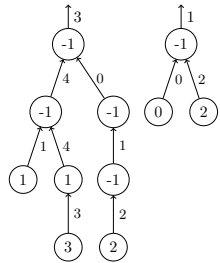
\begin{figure}[ht]
\centering
\scalebox{0.6}{\begin{tikzpicture}
\node[circle,draw](T0) at (1.0, 0) {-1};
\node[circle,draw](T00) at (0.5, -1.5) {-1};
\node[circle,draw](T000) at (0.0, -3.0) {1};
\draw[->](T000) -- node[right]{\small{1}} (T00);
\node[circle,draw](T001) at (1.0, -3.0) {1};
\node[circle,draw](T0010) at (1.0, -4.5) {3};
\draw[->](T0010) -- node[right]{\small{3}} (T001);
\draw[->](T001) -- node[right]{\small{4}} (T00);
\draw[->](T00) -- node[right]{\small{4}} (T0);
\node[circle,draw](T01) at (2.0, -1.5) {-1};
\node[circle,draw](T010) at (2.0, -3.0) {-1};
\node[circle,draw](T0100) at (2.0, -4.5) {2};
\draw[->](T0100) -- node[right]{\small{2}} (T010);
\draw[->](T010) -- node[right]{\small{1}} (T01);
\draw[->](T01) -- node[right]{\small{0}} (T0);
\node(Q0) at (1.0,1) {};
\draw[->] (T0) -- node[right]{3} (Q0);
\node[circle,draw](T1) at (3.5, 0) {-1};
\node[circle,draw](T10) at (3.0, -1.5) {0};
\draw[->](T10) -- node[right]{\small{0}} (T1);
\node[circle,draw](T11) at (4.0, -1.5) {2};
\draw[->](T11) -- node[right]{\small{2}} (T1);
\node(Q1) at (3.5,1) {};
\draw[->] (T1) -- node[right]{1} (Q1);
\end{tikzpicture}}
\caption{A flow on a forest of rooted ordered trees. The exit rate is 4.}
\label{fig:flow-example}
\end{figure}

The combinatorics of flows appears in the context of the Pre-Lie operad in \cite{Flows}. A formal power series $\varepsilon_F(t)$ can be associated with each forest $F$ by setting
\begin{equation}
\label{eq:open-flows}
\varepsilon_F(t) = \sum_f t^{r(f)}
\end{equation}
where the sum is over the flows $f$ of $F$, and $r(f)$ is the exit rate of $f$.  An inductive formula to compute this serie has been given in \cite{Flows}. A very surprising result is that the same induction appears in a very different context on intervals of the Tamari order. Indeed, the recursive description of a polynomial counting the number of elements smaller than a given tree in the Tamari order has been given in \cite{Me_Tamari}. It actually corresponds to the formal series of some flow by a simple change of variable $x = \frac{1}{1 -t}$. By taking the series at $t=0$, we obtain the following result.

\vspace{-0.2cm}

\begin{Theoreme}
The number of closed flows of a given forest $F$ is the number of elements smaller than or equal to a certain binary tree $T(F)$ in the Tamari order.
\end{Theoreme}

\vspace{-0.2cm}

The binary tree is obtained by a very classical bijection between forests of ordered trees and binary trees. The forest $F$ is actually the final forest of the binary tree $T(F)$, see Figure \ref{fig:forest-flow-list} for an example. This theorem can be proved by comparing the recursive formulas of \cite{Flows} and \cite{Me_Tamari} but our purpose here is to give an explicit bijection. More precisely, the bijection is defined between closed flows of forests and interval-posets. The forest itself gives the increasing relation of the interval-poset. The decreasing relations are then obtained from the inputs of the flow.

\begin{figure}[ht]
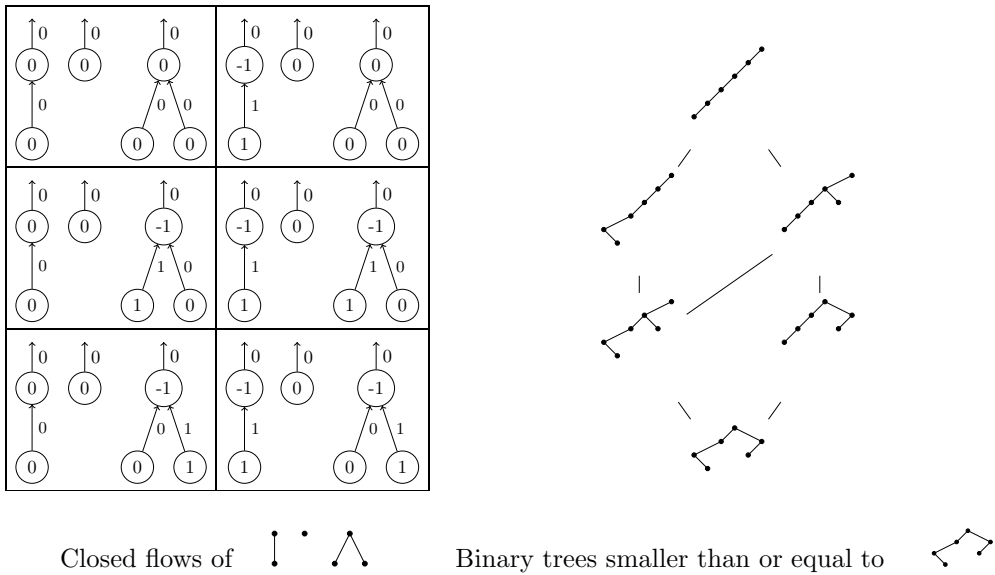

\centering
\def \wlev{2}
\def \hlev{2.8}
\def \sscale{0.6}
\def \fpath{figures/trees/}

\begin{tabular}{cc}
\scalebox{0.7}{
\begin{tabular}{|c|c|}
\hline
\begin{tikzpicture}
\node[circle,draw](T0) at (0.0, 0) {0};
\node[circle,draw](T00) at (0.0, -1.5) {0};
\draw[->](T00) -- node[right]{\small{0}} (T0);
\node(Q0) at (0.0,1) {};
\draw[->] (T0) -- node[right]{0} (Q0);
\node[circle,draw](T1) at (1.0, 0) {0};
\node(Q1) at (1.0,1) {};
\draw[->] (T1) -- node[right]{0} (Q1);
\node[circle,draw](T2) at (2.5, 0) {0};
\node[circle,draw](T20) at (2.0, -1.5) {0};
\draw[->](T20) -- node[right]{\small{0}} (T2);
\node[circle,draw](T21) at (3.0, -1.5) {0};
\draw[->](T21) -- node[right]{\small{0}} (T2);
\node(Q2) at (2.5,1) {};
\draw[->] (T2) -- node[right]{0} (Q2);
\end{tikzpicture} &
\begin{tikzpicture}
\node[circle,draw](T0) at (0.0, 0) {-1};
\node[circle,draw](T00) at (0.0, -1.5) {1};
\draw[->](T00) -- node[right]{\small{1}} (T0);
\node(Q0) at (0.0,1) {};
\draw[->] (T0) -- node[right]{0} (Q0);
\node[circle,draw](T1) at (1.0, 0) {0};
\node(Q1) at (1.0,1) {};
\draw[->] (T1) -- node[right]{0} (Q1);
\node[circle,draw](T2) at (2.5, 0) {0};
\node[circle,draw](T20) at (2.0, -1.5) {0};
\draw[->](T20) -- node[right]{\small{0}} (T2);
\node[circle,draw](T21) at (3.0, -1.5) {0};
\draw[->](T21) -- node[right]{\small{0}} (T2);
\node(Q2) at (2.5,1) {};
\draw[->] (T2) -- node[right]{0} (Q2);
\end{tikzpicture} \\ \hline
\begin{tikzpicture}
\node[circle,draw](T0) at (0.0, 0) {0};
\node[circle,draw](T00) at (0.0, -1.5) {0};
\draw[->](T00) -- node[right]{\small{0}} (T0);
\node(Q0) at (0.0,1) {};
\draw[->] (T0) -- node[right]{0} (Q0);
\node[circle,draw](T1) at (1.0, 0) {0};
\node(Q1) at (1.0,1) {};
\draw[->] (T1) -- node[right]{0} (Q1);
\node[circle,draw](T2) at (2.5, 0) {-1};
\node[circle,draw](T20) at (2.0, -1.5) {1};
\draw[->](T20) -- node[right]{\small{1}} (T2);
\node[circle,draw](T21) at (3.0, -1.5) {0};
\draw[->](T21) -- node[right]{\small{0}} (T2);
\node(Q2) at (2.5,1) {};
\draw[->] (T2) -- node[right]{0} (Q2);
\end{tikzpicture} &
\begin{tikzpicture}
\node[circle,draw](T0) at (0.0, 0) {-1};
\node[circle,draw](T00) at (0.0, -1.5) {1};
\draw[->](T00) -- node[right]{\small{1}} (T0);
\node(Q0) at (0.0,1) {};
\draw[->] (T0) -- node[right]{0} (Q0);
\node[circle,draw](T1) at (1.0, 0) {0};
\node(Q1) at (1.0,1) {};
\draw[->] (T1) -- node[right]{0} (Q1);
\node[circle,draw](T2) at (2.5, 0) {-1};
\node[circle,draw](T20) at (2.0, -1.5) {1};
\draw[->](T20) -- node[right]{\small{1}} (T2);
\node[circle,draw](T21) at (3.0, -1.5) {0};
\draw[->](T21) -- node[right]{\small{0}} (T2);
\node(Q2) at (2.5,1) {};
\draw[->] (T2) -- node[right]{0} (Q2);
\end{tikzpicture} \\ \hline
\begin{tikzpicture}
\node[circle,draw](T0) at (0.0, 0) {0};
\node[circle,draw](T00) at (0.0, -1.5) {0};
\draw[->](T00) -- node[right]{\small{0}} (T0);
\node(Q0) at (0.0,1) {};
\draw[->] (T0) -- node[right]{0} (Q0);
\node[circle,draw](T1) at (1.0, 0) {0};
\node(Q1) at (1.0,1) {};
\draw[->] (T1) -- node[right]{0} (Q1);
\node[circle,draw](T2) at (2.5, 0) {-1};
\node[circle,draw](T20) at (2.0, -1.5) {0};
\draw[->](T20) -- node[right]{\small{0}} (T2);
\node[circle,draw](T21) at (3.0, -1.5) {1};
\draw[->](T21) -- node[right]{\small{1}} (T2);
\node(Q2) at (2.5,1) {};
\draw[->] (T2) -- node[right]{0} (Q2);
\end{tikzpicture} &
\begin{tikzpicture}
\node[circle,draw](T0) at (0.0, 0) {-1};
\node[circle,draw](T00) at (0.0, -1.5) {1};
\draw[->](T00) -- node[right]{\small{1}} (T0);
\node(Q0) at (0.0,1) {};
\draw[->] (T0) -- node[right]{0} (Q0);
\node[circle,draw](T1) at (1.0, 0) {0};
\node(Q1) at (1.0,1) {};
\draw[->] (T1) -- node[right]{0} (Q1);
\node[circle,draw](T2) at (2.5, 0) {-1};
\node[circle,draw](T20) at (2.0, -1.5) {0};
\draw[->](T20) -- node[right]{\small{0}} (T2);
\node[circle,draw](T21) at (3.0, -1.5) {1};
\draw[->](T21) -- node[right]{\small{1}} (T2);
\node(Q2) at (2.5,1) {};
\draw[->] (T2) -- node[right]{0} (Q2);
\end{tikzpicture} \\ \hline
\end{tabular}}
&
\scalebox{0.6}{
\begin{tikzpicture}[baseline=-3.5cm]

\node(T1) at (0,0){
\scalebox{\sscale}{
	\input{\fpath T6-ex1}
}
};

\node(T2) at (-1 * \wlev, -1 * \hlev) {
\scalebox{\sscale}{
	\input{\fpath T6-ex2}
}
};

\node(T3) at (1 * \wlev, -1 * \hlev) {
\scalebox{\sscale}{
	\input{\fpath T6-ex3}
}
};

\node(T4) at (-1 * \wlev,-2 * \hlev) {
\scalebox{\sscale}{
	\input{\fpath T6-ex4}
}
};

\node(T5) at (1 * \wlev, -2 * \hlev){
\scalebox{\sscale}{
	\input{\fpath T6-ex5}
}
};

\node(T6) at (0, -3 * \hlev) {
\scalebox{\sscale}{
	\input{\fpath T6-ex6}
}
};

\draw(T1) -- (T2);
\draw(T1) -- (T3);
\draw(T2) -- (T4);
\draw(T3) -- (T4);
\draw(T3) -- (T5);
\draw(T4) -- (T6);
\draw(T5) -- (T6); 
\end{tikzpicture}}
\\ 
Closed flows of~~~
\scalebox{0.4}{
\begin{tikzpicture}
\node(T0) at (0,0){};
\node(T00) at (0,-1){};
\node(T1) at (1,0){};
\node (T2) at (2.5,0){};
\node(T20) at (2,-1){};
\node(T21) at (3,-1){};
\draw[Point] (T0) circle;
\draw[Point] (T00) circle;
\draw[Point] (T1) circle;
\draw[Point] (T2) circle;
\draw[Point] (T20) circle;
\draw[Point] (T21) circle;
\draw (T00.center) -- (T0.center);
\draw (T20.center) -- (T2.center);
\draw (T21.center) -- (T2.center);
\end{tikzpicture}}

&

Binary trees smaller than or equal to ~~~
\scalebox{0.3}{
	\input{\fpath T6-ex6}
}

\end{tabular}
\caption{Flows of a forest and Tamari ideal.}
\label{fig:forest-flow-list}
\end{figure}

\vspace{-0.2cm}

\subsection{Bijection between flows and interval-posets}

\vspace{-0.1cm}

The first step of the bijection consists in labelling the nodes of the forest. The labelling is done recursively starting with the root followed by its children from left to right. The labelled nodes become the vertices of the interval-poset, see Figure \ref{fig:flows-bijection} for an example. 

We then add the increasing relations of the interval-posets. These relations depend only on the forest itself and not on its actual flow. For each vertex $i$, we add a relation $i \trprec j$ where $j>i$ is the first vertex which is not a 	descendant of $i$. Equivalently, if $i$ has a right brother $j$, we add all relations $i' \trprec j$ where $i'$ runs over all the nodes of the right most branch of $i$. This is illustrated on the first image of Figure \ref{fig:flows-bijection}. 

Finally, we gradually add the decreasing relations. The process is illustrated on Figure \ref{fig:flows-bijection}. At each step, we deal with one of the negative inputs. We take the inputs in the decreasing order of their corresponding labels in the interval-poset (from the first step of the bijection). The \emph{source} of a negative input is the first strictly positive input of its descendants (still following the label order). As an example, on the fourth image of Figure \ref{fig:flows-bijection}, the source of the selected negative input is its left child (labelled 3) and not its right child (labelled 4). For a negative input labelled $i$ with $j$ as a source, we then add all decreasing relations $j' \trprec i$ for all $i < j' \leq j$.  

\vspace{-0.1cm}

\begin{Proposition}
The previously described process is well-defined and gives a bijection between flows of ordered forests and interval-posets.
\end{Proposition}

\vspace{-0.3cm}

\begin{proof}
The first property to check is that the constructed object is indeed an interval-poset. This is true by construction. A decreasing relation $j \trprec i$ can never be added if we already had $i \trprec j$. Indeed, $j \trprec i$ means that $j$ is a descendant of $i$ in the forest and $i \trprec j$ means $j$ is not. Furthermore, it is easy to check that when an increasing relation $i \trprec j$ is added, then all relation $i' \trprec j$ where $i \leq i' < j$ are also added and so the final object satisfies the interval-poset conditions.

To prove that this process is actually a bijection, we need to describe the inverse process to obtain a flow from an interval-poset. First, we have to construct the forest from the increasing relation. This is simply the inverse process of what we described earlier: the parent of a node $j$ is the largest number $i<j$ such that $i \ntrprec j$. Then we have to add the inputs of the flow. Each vertex $i$ such that there exists $j>i$ with $j \trprec i$ receives a $-1$ input and increases the input of a source. Its source is the biggest vertex $j>i$ with $j \trprec i$. Note that a vertex cannot be both a $-1$ input and a source because if $j' \trprec j \trprec i$ with $i<j<j'$, then $j$ cannot be the source of $i$. By a step by step proof, it is clear that this process reverses the one we described earlier. 
\end{proof}

\begin{figure}[b!]
\centering
\scalebox{0.7}{\input{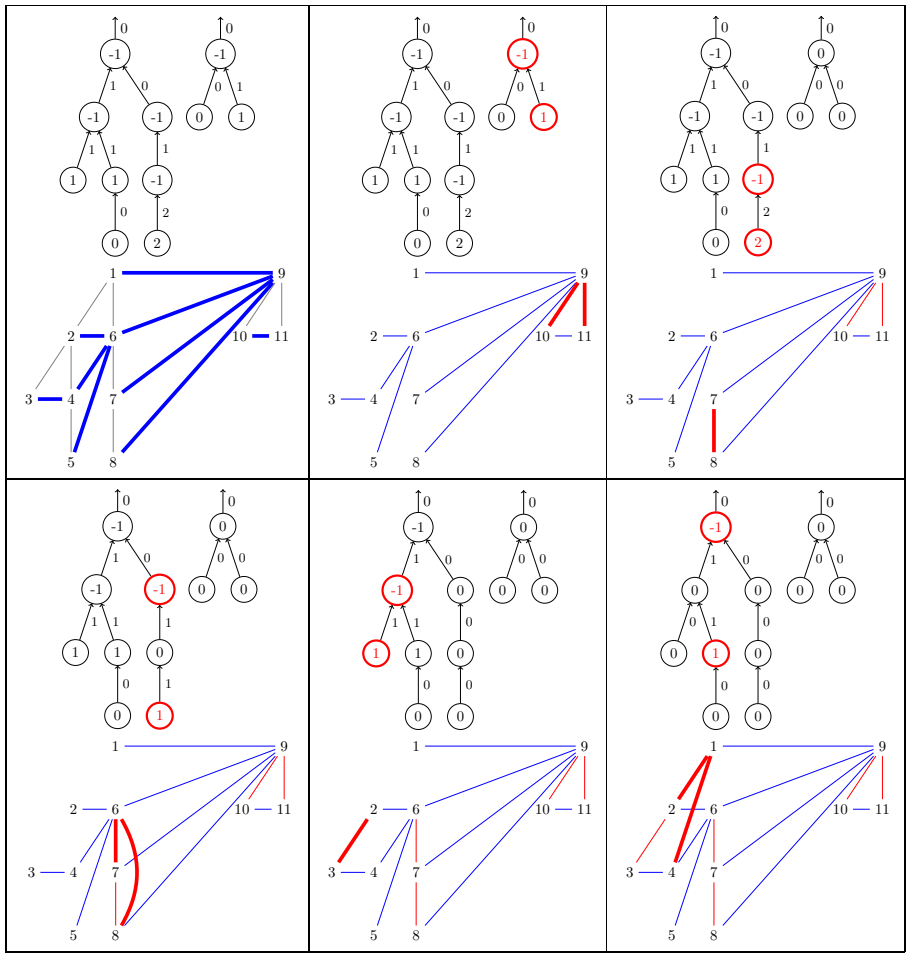}}
\caption{Bijection between flows and interval-posets}
\label{fig:flows-bijection}
\end{figure}

\vspace{-0.3cm}

\subsection{Statistics and open flows}

\vspace{-0.1cm}

Some statistics can be read on both flows and interval-posets. An easy one would be the number of $-1$ inputs on the flow. They trivially correspond to the number of vertices $a$ of the interval-poset such that there is a relation $a+1 \trprec a$. One can also compute the sum of all outgoing rates of non-roots nodes. This is equal to 7 on the Figure \ref{fig:flows-bijection} example. This can also be read on the interval-poset. For each node $a$, we take the set of vertices $\lbrace b > a; b \trprec a; \forall c \trprec a, b \ntrprec c \rbrace$. In other words, these are the maximal elements in terms of increasing relations which precede $a$ with a decreasing relation. As an example, on Figure \ref{fig:flows-bijection}, we obtain $\lbrace 2, 4 \rbrace$ for the vertex 1, $\lbrace 3 \rbrace$ for 2, $\lbrace 7, 8 \rbrace$ for 6, $\lbrace 8 \rbrace$ for 7, and $\lbrace 11 \rbrace$ for 9. By summing all the sizes, we obtain $7$ which is the sum of outgoing rates. 

It is possible to prove by induction that the series of open flows of a given forest \eqref{eq:open-flows} is actually a polynomial in $\frac{1}{1 - t}$. It corresponds to the Tamari polynomial defined in \cite{Me_Tamari}, the number of terms is the number of closed flows of the forest. This can also be explained from a combinatorial point of view. Each open flow can be sent to a unique closed flow. The serie of open flows corresponding to a closed flow $f$ is then a monomial $\left( \frac{1}{1 -t} \right)^r$ where $r$ is equal to $\Itrees(I)$ and $I$ is is the image interval-poset of $f$. We will discuss this further in some future work.

\begin{footnotesize}
  \bibliographystyle{plain}
  \label{sec:biblio}
  \bibliography{arxiv}
\end{footnotesize} 

\end{document}